\documentclass[10pt,a4paper,reqno]{amsart}

\usepackage{amsmath,amscd,mathrsfs,amsthm,amsxtra,mathtools,amssymb}

\usepackage[usenames,dvipsnames,svgnames,table]{xcolor}

\usepackage{url,hyperref}

\usepackage[all]{xy}

\usepackage{fullpage}

\usepackage[parfill]{parskip} 

\usepackage{graphicx}
\usepackage{ifpdf}
\ifpdf
   \fi

\theoremstyle{plain}
\newtheorem{thm}{Theorem}[section]
\newtheorem{lem}[thm]{Lemma}
\newtheorem{cor}[thm]{Corollary}
\newtheorem{prop}[thm]{Proposition}
\newtheorem*{maina}{Theorem A}
\newtheorem*{mainb}{Theorem B}

\theoremstyle{definition}
\newtheorem{dfn}[thm]{Definition}

\DeclareMathOperator{\aug}{aug}
\DeclareMathOperator{\coker}{coker}
\DeclareMathOperator{\ev}{ev}
\DeclareMathOperator{\im}{im}
\DeclareMathOperator{\Hom}{Hom}

\def\twist{\mathbf{\cdot}}

\title{Poincar{\'{e}} Duality Complexes with Highly Connected Universal Covers}
\author{B.Bleile and I.Bokor}
\date{September, 2015}

\usepackage{soul}

\begin{document}

\maketitle

\begin{abstract}

\noindent Baues and Bleile showed that, up to oriented homotopy equivalence, a Poincar{\'{e}} duality complex of dimension $n \ge 3$ with $(n-2)$-connected universal cover, is classified by its fundamental group, orientation class and the image of its fundamental class in the homology of the fundamental group. We generalise Turaev's results for the case $n=3$, by providing necessary and sufficient conditions for a triple $(G, \omega, \mu)$, comprising a group, $G$, $\omega \in H^{1}\left(G; \raisebox{0.3ex}{$\mathbb{Z}$}/\raisebox{-0.2ex}{$2\mathbb{Z}$}\right)$ and $\mu \in H_{3}(G; \mathbb{Z}^{\omega})$ to be realised by a Poincar{\'{e}} duality complex of dimension $n$ with $(n-2)$-connected universal cover, and  by showing that such a complex is a connected sum of two such complexes if and only if its fundamental group is a free product of groups.
\end{abstract}

\section{Introduction}\label{sect:intro}

Important properties of manifolds, especially global ones such as \emph{Poincar{\'{e}} duality}, depend only on their homotopy properties. This realisation led to the notion of \emph{Poincar{\'{e}} duality complexes} as the homotopy theoretic generalisation of manifolds, opening the possibility of classifying manifolds using a classification of Poincar{\'{e}} duality complexes.

Since not every Poincar{\'{e}} duality complex arises from a manifold, such complexes are genuinely more general. In this case, the greater generality is not an additional complication, but a convenience, for Poincar{\'{e}} duality complexes are accessible to algebraic analysis, thereby allowing more computation. This has been successfuly exploited by a number of authors to achieve partial classification using algebraic invariants.

In particular, Hendricks classified Poincar{\'{e}} duality complexes of dimension $3$ (\emph{PD$^{3}$-complexes}) up to oriented homotopy equivalence using the ``fundamental triple''\!,  comprising  the fundamental group, the orientation character and the image of the fundamental class in the cohomology of the fundamental group.

Turaev (\cite{VGT1990}) gave an alternative proof of Hendricks' result and   provided necessary and sufficient conditions for a triple $(G, \omega, \mu)$, consisting of a group, $G$, a cohomology class, $\omega \in H^{1}\!\left(G ; \raisebox{0.3ex}{$\mathbb{Z}$}/\raisebox{-0.2ex}{$2\mathbb{Z}$}\right)$, and a homology class, $\mu \in H_{3}(G ; \mathbb{Z}^{\omega})$, where $\mathbb{Z}^{\omega}$ denotes the integers, $\mathbb{Z}$, regarded as a right $\mathbb{Z}[G]$-module with respect to the \emph{twisted} structure induced by $\omega$, to be fundamental triple of a PD$^{3}$-complex with simply connected universal covering space. Central to this was that applying a specific function, which we call \emph{the Turaev map}, to $\mu$, yield isomorphism in the stable category of $\mathbb{Z}[G]$, that is to say, a homotopy equivalence of $\mathbb{Z}[G]$-modules. 

Turaev conjectured that his result holds for all PD$^{n}$-complexes with $(n-2)$-connected universal cover.

Baues and Bleile  classified Poincar{\'{e}} duality complexes of dimension 4 in \cite{HJB-BB2010}. Their analysis showed that a Poincar{\'{e}} duality complex, $X$,  of dimention $n\ge 3$, is classified up to oriented homotopy equivalence by the triple comprising its $(n-2)$-type $P_{n-2}(X)$, its orientation character $\omega_{X} \in H^{1}\left(\pi_{1}(X) ; \,  \raisebox{0.3ex}{$\mathbb{Z}$}/\raisebox{-0.2ex}{$2\mathbb{Z}$}\right)$ and its fundamental class, $[X] \in H_{n}(X; \mathbb{Z}^{\omega_{X}})$.

They called this the \emph{fundamental triple of $X$}, as it is a generalisation of Hendricks' ``fundamental triple'', for the $(n-2)$ type of $X$ determines its fundamental group. Moreover, when the universal cover of the complex is $(n-2)$-connected --- automatically the case when $n=3$  ---  the $(n-2)$ section is an Eilenberg-Mac Lane space of type $(\pi_{1}(X),1)$, so that the $(n-2)$ type  and the fundamental group determine each other completely, reducing their fundamental triple to that of Hendricks.

It follows that two PD$^{n}$-complexes with $(n-2)$-connected universal cover are homotopically equivalent if and only if their fundamental triples are isomorphic, verifying the first part of Turaev's conjecture. 

Our main results complete a proof of Turaev's conjecture, which we formulate in two theorems.

Recall that the group $G$ is of type FP$_{n}$ if and only if the trivial $\mathbb{Z}[G]$-module $\mathbb{Z}$ has a projective resolution, $\mathbf{P}$, with $P_{j}$ finitely generated for $j \le n$.

\begin{maina}
  Let $G$ be a finitely presentable group, $\omega$ a cohomology class in $H^{1}\left(G; \, \raisebox{0.5ex}{$\mathbb{Z}$}/\raisebox{-0.3ex}{$2\mathbb{Z}$}\right)$, and $\mu$ a homology class in $H_{n}(G; \mathbb{Z}^{\omega})$.
  
  If $G$ is of type FP$_{n}$ with $H^{i}(G; \,^{\omega}\! \mathbb{Z}[G]) =0$ for $1 < i < n-1$, and  $n \ge 3$, then $(G, \omega, \mu)$ can be realised as the fundamental triple of a PD$^{n}$-complex with $(n-2)$-connected universal cover if and only if the Turaev map applied to $\mu$ is an isomorphism in the stable category of $\mathbb{Z}[G]$.
\end{maina}

\begin{mainb}
A PD$^{n}$-complex with $(n-2)$-connected universal cover, decomposes as  connected sum if and only if its fundamental group decomposes as a free product of groups.
\end{mainb}

Section \ref{sect:notation} summarises background material and fixes notation.

Section \ref{sect:necessity} contains the formulation and proof of the necessity of the condition for the realisation of a  fundamental triple by a PD$^{n}$-complex with $(n-2)$-connected universal cover.

Section \ref{sect:sufficiency} completes the proof of Theorem A, with the sufficiency of the condition in Section \ref{sect:necessity}.

Section \ref{sect:connsum} contains the proof of Theorem B, showing how the fundamental triple detects connected sums.

\section{Background and Notation}\label{sect:notation}

This section summarises background material and fixes notation for the rest of the paper. Details and further references can be found in \cite{BB2010}.

\subsection{}\label{subsection:algebra}
 Let $\Lambda$ be the integral group ring, $\mathbb{Z}[G]$, of the finitely presented group, $G$.

We write $I$ for the \emph{augmentation ideal}, the kernel of the augmentation map 
\begin{equation*}\label{eqn:aug}
   \aug \colon \Lambda \longrightarrow \mathbb{Z}, \quad \sum_{g \in G} n_{g} g \longmapsto \sum_{g \in G} n_{g} 
\end{equation*}
where $\mathbb{Z}$ is a $\Lambda$-bi-module with trivial $\Lambda$ action.

Each cohomology class $\omega \in H^{1}\left(G; \, \raisebox{0.5ex}{$\mathbb{Z}$}/\raisebox{-0.3ex}{$2\mathbb{Z}$}\right)$ may be viewed as a group homomorphism 
\[
\omega \colon G \longrightarrow  \raisebox{0.3ex}{$\mathbb{Z}$}/\raisebox{-0.2ex}{$2\mathbb{Z}$} =\{ 0, 1\}
\]
and yields an anti-isomorphism
\begin{equation*}
   \overline{\phantom{B}} \colon \Lambda \longrightarrow \Lambda, \quad \lambda= \sum_{g \in G} n_{g} g \longmapsto \overline{\lambda} = \sum_{g \in G} (-1)^{\omega(g)}  n_{g} g^{-1}  
   \end{equation*}

Consequently, a right $\Lambda$-module, $A$, yields the left $\Lambda$-module, $^{\omega}\!A$, with action given by
\begin{equation*}\label{eqn:lectactionfromright}
   \lambda  \twist a := a .\overline{\lambda}
\end{equation*}
for $\lambda \in \Lambda, a \in A$. 

Analogously, a left $\Lambda$-module $B$ yields the right $\Lambda$-module, $B^{\omega}$.

Given left $\Lambda$-modules $A_{j}, B_{i}$ for $ 1 \le i \le k$ and $1 \le j \le \ell$, we sometimes write the $\Lambda$-morphism 
\[
\psi \colon \bigoplus_{j=1}^{\ell} A_{j} \longrightarrow \bigoplus_{i=1}^{k}B_{i}
\]
in matrix form as
\[
   \bigoplus_{j=1}^{\ell} A_{j} \xrightarrow[]{ \ \begin{bmatrix}\psi_{ij} \end{bmatrix}_{k \times \ell}\ \ } \bigoplus_{i=1}^{k}B_{i}
\]
for $\psi_{ij} = pr_{i} \circ \psi \circ in_{j} \colon A_{j} \longrightarrow B_{i}$, where $in_{j}$ is the $j^{\text{th}}$ natural inclusion and $pr_{i}$ the $i^{\text{th}}$ natural projection of the direct sum. The composition of such morphisms is given by matrix multiplication.

If $B$ is a left $\Lambda$-module and $M$  a $\Lambda$-bi-module, then $\Hom_{\Lambda}(B,M)$ is a right $\Lambda$-module with  action given by
\begin{equation*}
  \varphi.\lambda \colon B \longrightarrow M, \quad b \longmapsto \varphi(b).\lambda
\end{equation*}

The dual of the left $\Lambda$-module $B$ is the left $\Lambda$-module $B^{\ast} = \,^{\omega}\!\Hom_{\Lambda}(B, \Lambda)$.

This defines an endofunctor on the category of left $\Lambda$-modules.

Evaluation  defines a natural transformation, $^{\omega}\!\varepsilon$, from the identity functor to the double dual functor, where for the left $\Lambda$-module $B$, 
\[
  ^{\omega}\varepsilon_{B} \colon B \longrightarrow  B^{\ast\ast} = \, ^{\omega}\!\Hom_{\Lambda}\big(\,^{\omega}\!\Hom_{\Lambda}(B, \Lambda), \Lambda\big), \quad b \longmapsto \, \overline{\ev_{b}}
\]
with $\overline{\ev_{b}}$  defined by
\[
    \overline{\ev_{b}} \colon \, ^{\omega}\!\Hom(B, \Lambda) \longrightarrow \Lambda, \quad \psi \longmapsto \overline{\psi(b)}
\]

The left $\Lambda$-module, $A$, defines the natural transfomation, 
\(
\eta
\),  from the functor \(A^{\omega} \otimes_{\Lambda} \underline{\phantom{B}} \) to the functor \( \Hom_{\Lambda}\big(\,^{\omega}\!\Hom_{\Lambda}(\,\underline{\phantom{B}} \, , \Lambda), A\big) \) where, for the left $\Lambda$-module $B$, 
\[
   \eta_{B} \colon A^{\omega} \otimes_{\Lambda} B \longrightarrow \Hom_{\Lambda}(B^{\ast}\! , A) = \Hom_{\Lambda}\big(\,^{\omega}\!\Hom_{\Lambda}(B, \Lambda), A\big)
\]is given by
\[
   \eta_{B}(a \otimes b ) \colon \psi \longmapsto \overline{\psi(b)}.a
\]for $a \otimes b \in A^{\omega}\otimes_{\Lambda} B$.

The $\Lambda$-morphisms $f,g \colon A_{1} \longrightarrow A_{2}$ are \emph{homotopic} if and only if the homomorphism
\[
    f-g \colon A_{1} \longrightarrow A_{2}, \quad x \longmapsto f(x) - g(x)
\]
factors through a projective $\Lambda$-module $P$.

Associated with $\Lambda$ is its \emph{stable category}, whose objects are all $\Lambda$-modules  and whose morphisms are all homotopy classes of $\Lambda$-morphisms. Thus,  an isomorphism in the stable category of $\Lambda$ is a homotopy equivalence of $\Lambda$-modules.

\subsection{}\label{subsection:CW|}

We work in the category of connected well pointed $CW$-complexes and pointed maps.  We write $X^{[k]}$ for the $k$-skeleton of $X$, suppressing the base point from our notation.

The inclusion of the $k$-skeleton into the $(k+1)$-skeleton induces homomorphisms
\[
    \iota_{k, r} \colon \pi_{r}(X^{[k]}) \longrightarrow \pi_{r}(X^{[k+1]})
\]
and we write $\Gamma_{k+1}(X)$ for $\im(\iota_{k,k+1})$.

From now, we work with the fundamental group of $X$, and its integral group ring
\begin{equation*}\label{eqn:gpring}
    \Lambda = \mathbb{Z}[\pi_{1}(X)]
\end{equation*}

We take $X$ to be a reduced $CW$-complex, so that $X^{[0]} = \{\ast\}$, and write 
\begin{equation*}\label{eqn:univcover}
     u \colon \widetilde{X}\longrightarrow X
\end{equation*}
for the universal cover of $X$, fixing a base point for  $\widetilde{X}$ in $u^{-1}(\ast)$.

We write  $\mathbf{C}(\widetilde{X})$ for  the cellular chain complex of $\widetilde{X}$ viewed as a complex of left $\Lambda$-modules. 

Since $X$ is reduced, $C_{0}(\widetilde{X}) = \Lambda$.

The homology and cohomology of $X$ we work with are the abelian groups
\begin{align*}
   \widehat{H}_{q}(X;A) &:= H_{q}(A \otimes_{\Lambda} \mathbf{C}(\widetilde{X}))\\
   \widehat{H}^{q}(X;B) &:= H^{q}\big(\Hom_{\Lambda}(\mathbf{C}(\widetilde{X}), B))
\end{align*}
where $A$ is a right $\Lambda$-module and $B$ is a left $\Lambda$-module.

We write $H_{-q}(X;B)$ for $H^{q}(X;B)$, when treating cohomology as ``homology in negative degree'' simplifies the exposition.

\subsection{}\label{subsection:pdncomplex}

An \emph{$n$-dimensional Poincar{\'e} duality complex} (PD$^n$-complex) comprises a reduced connected $CW$-complex, $X$,  whose  fundamental group, $\pi_{1}(X)$,  is finitely presented, together with an \emph{orientation character}, $\omega_{X} \in H^{1}\left(\pi_{1} (X) ; \,  \raisebox{0.3ex}{$\mathbb{Z}$}/\raisebox{-0.2ex}{$2\mathbb{Z}$} \right)$, viewed as a group homomorphism 
\( 
 \pi_{1} (X) \longrightarrow  \raisebox{0.3ex}{$\mathbb{Z}$}/\raisebox{-0.2ex}{$2\mathbb{Z}$}
 \), 
and a \emph{fundamental class }$[X] \in H_n(X; \mathbb Z^{\omega_{X}})$, such that  for every $r \in \mathbb{Z}$ and left $\mathbb{Z} [\pi_{1} (X)]$-module $M$, the \emph{cap product with $[X]$},
\[ 
\underline{\phantom{A}}\smallfrown [X] \colon  H^{r}(X; M) \longrightarrow H_{n-r}(X; M^{\omega_{X}}),\quad \alpha \longmapsto \alpha \smallfrown  [X]\]
is an isomorphism of abelian groups.

We denote this by $(X, \omega_{X}, [X])$.

Wall  (\cite{W1965}, \cite{W1967}) showed   that for $n >3$, every PD$^{n}$ complex is \emph{standard}, meaning that it is homotopically equivalent to an $n$-dimensional $CW$-complex with precisely one $n$-cell, whereas a PD$^{3}$ complex, $X$, is either standard, or \emph{weakly standard}, meaning that it is homotopically equivalent to one of the form $X^{\prime} \cup e^{3}$, where $e^{3}$ is a 3-cell and $X^{\prime}$ is a 3-dimensional $CW$-complex with $H^{n}(X^{\prime};B) = 0$ for all coefficient modules $B$.

\subsection{Homotopy Systems}\label{subsection:homotopysystems}

Baues introduced \emph{homotopy systems} to investigate when chain complexes and chain maps of free $\Lambda$-modules are realised by $CW$-complexes (\cite{HJB-1991}).

\begin{dfn}\label{dfn:homotopysystem}
  Let $n$ be an integer greater than 1. A \emph{homotopy system of order $(n+1)$} comprises 
  \begin{itemize}
  
  \item[(a)] a reduced $n$-dimensional $CW$-complex, $X$,
  
  \item[(b)]  a chain complex $\mathbf{C}$ of free $\Lambda:=\mathbb{Z}[\pi_{1}(X)]$-modules which coincides with $\mathbf{C}(\widetilde{X}) $ in degree $q$ whenever $q \le n$,
  
  \item[(c)] a homomorphism, \( f_{n+1} \colon C_{n+1} \longrightarrow \pi_{n}(X) \) with  \( f_{n+1} \circ d_{n+2} = 0 \),  which renders commutative the diagram
  \[
     \xymatrix{
         C_{n+1} \ar@{->}[r]^{\ f_{n+1}\ } \ar@{->}[d]_{ d_{n+1} \ } & \pi_{n}(X) \ar@{->}[d]^{ \ j} \\
         C_{n} & \pi_{n} (X, X^{[n-1]})\ar@{->}[l]^{h_{n} \qquad } 
     }
  \]
    where $j$ is induced by the inclusion $(X, \ast) \longrightarrow (X, X^{[n-1]})$, and
  \[
     h_{n} \colon \pi_{n}(X, X^{[n-1]}) \xrightarrow[\cong]{ \  \ u_{\ast}^{-1} \ \ } \pi_{n}(\widetilde{X}, \widetilde{X}^{[n-1]}) \xrightarrow[\cong]{ \  \ h \ \ } H_{n}(\widetilde{X}, \widetilde{X}^{[n-1]}) 
  \]
  is  the Hurewicz isomorphism, $h$, composed with $u^{-1}_{\ast}$, the inverse of the isomorphism induced by the universal covering map.

  \end{itemize}
\end{dfn}

\section{Formulation and Necessity of the Realisation Condition}\label{sect:necessity}

\subsection{}

As the $(n-2)$-type, $P_{n-2}X$, of the PD$^{n}$-complex, $X$, with $(n-2)$-connected universal cover, is an Eilenberg-Mac Lane space, $K\big(\pi_{1}(X), 1\big)$,  we may identify the fundamental triple of $X$ with $\big( \pi_{1}(X), \omega, \mu\big)$, where $\omega$ and $\mu$ are, respectively,  the images of $\omega_{X}$ and $[X]$ in the group homology of $\pi_{1}(X)$.

\begin{lem}\label{lem:Hi=0}
    Let $(X, \omega_{X}, [X])$ be a PD$^{n}$-complex with $(n-2)$-connected universal cover. Then, for all $1 \le i < n-1$,
    \[
         H^{i}\big(\pi_{1}(X) ; \,^{\omega}\!\Lambda \big) = 0
    \]
\end{lem}

\begin{proof}
   Construct an Eilenberg-Mac Lane space $K = K\big(\pi_{1}(X), 1\big)$ from $X$ by attaching cells of dimension $n$ and higher.
   
   As the universal cover $\widetilde{X}$ of $X$ is $(n-2)$-connected, the cellular chain complexes of the universal covers $\widetilde{X}$ and $\widetilde{K}$ coincide in degrees below $n$:
   \begin{align*}
       C_{i}(\widetilde{X}) = C_{i}(\widetilde{K}) && \text{for } 0 \le i < n
   \end{align*}
   
   Hence, for $ 1 < i < n-1$,
   \begin{align*}
        H^{i}\big(\pi_{1}(X) ; \,^{\omega}\!\Lambda \big)  & = H^{i}(\widetilde{K} ; \,^{\omega}\!\Lambda) \\ 
                       & = H^{i}(\widetilde{X}; \,^{\omega}\!\Lambda) \\ 
                       & \cong  H_{n-i}(\widetilde{X} ; \Lambda) &&\text{by Poincar{\'{e}} duality}\\ 
                       & = 0
   \end{align*}
\end{proof}

\subsection{}
To formulate the realisation condition, we introduce set of functors, $\{G_{q} \mid q \in \mathbb{Z}\}$, from the category of chain complexes of projective left $\Lambda$-modules and homotopy classes of chain maps to the stable category of $\Lambda$.

Given chain complexes of projective left $\Lambda$-modules, $\mathbf{C}, \mathbf{D}$, and a chain map $\mathbf{f} \colon \mathbf{C} \longrightarrow \mathbf{D}$, we have the commutative diagram
\[
 \xymatrix{
 \cdots \ \ar@{->}[r] & \ C_{q} \ \ar@{->}[d]_{f_{q}} \ar@{->}[r]^{\ d^{\mathbf{C}}_{q} \ \ }  \ & C_{q-1}\  \ar@{->}[d]^{f_{q-1}} \ar@{->}[r] &\cdots \\
 \cdots \ \ar@{->}[r]& \  D_{q} \ \ar@{->}[r]_{\ d^{\mathbf{D}}_{q} \ \ } & \ D_{q-1} \ar@{->}[r] &\cdots
 }
\]
Put
\begin{equation*}
   G_{q}(\mathbf{C}) : = \coker \left(d_{q}^{\mathbf{C}}\right) = \raisebox{0.3ex}{$C_{q-1}$}/\raisebox{-0.2ex}{$\im \left(d_{q}^{\mathbf{C}}\right)$}
\end{equation*}
and let $G_{q}(\mathbf{f})$ be the induced map of cokernels
\begin{equation*}
  \begin{split}
  \xymatrix{
  \im(d_{q}^{\,\mathbf{C}}) \ar@{->}[d] \quad  \ar@{>->}[r]& \ C_{q-1}\ar@{->}[d]_{f_{q-1}} \ \ar@{->>}[r] & \ G_{q}(\mathbf{C}) \ar@{..>}[d]_{\exists !}^{G_{q}(\mathbf{f})}\\ 
   \im(d_{q}^{\,\mathbf{D}}) \quad  \ar@{>->}[r]& \ D_{q-1} \ \ar@{->>}[r] & \ G_{q}(\mathbf{D}) 
  }
  \end{split}
\end{equation*}

Direct verification shows that each $G_{q}$ is a functor from the category of chain complexes of left $\Lambda$-modules to the category of left $\Lambda$-modules.

By Lemma 4.2 in \cite{BB2010}, chain homotopic maps $\mathbf{f} \simeq \mathbf{g} \colon \mathbf{C} \longrightarrow \mathbf{D}$ induce homotopic maps $G_{n}(\mathbf{f}) \simeq G_{n}(\mathbf{g})$, that is, $G_{q}(\mathbf{f}) - G_{q}(\mathbf{g})$ factors through a projective $\Lambda$-module. 

Hence, for each $q \in \mathbb{Z}$, $G_{q}$ is a functor from the category of chain complexes of projective left $\Lambda$-modules to the stable category of $\Lambda$.

\subsection{}

Let $X$ be PD$^{n}$ complex with $n \ge 3$. Put $\Lambda = \mathbb{Z}[\pi_{1}(X)]$ and take $\omega \in H^{1}\!\left(\pi_{1}(X);\, \raisebox{0.3ex}{$\mathbb{Z}$}/\raisebox{-0.2ex}{$2\mathbb{Z}$}\right)$.

 By Remark 2.3 and Lemma 3.6 in \cite{HJB-BB2010}, we may assume that $X = X^{\prime} \cup e^{n}$ is standard (or weakly standard if $n=3$) with
\begin{equation*}
   C_{n}(\widetilde{X}) = C_{n}(\widetilde{X^{\prime}}) \oplus \Lambda[e]
\end{equation*} 
where $[e]$  corresponds to $e^{n}$, $1 \otimes [e] \in \mathbb{Z}^{\omega} \otimes_{\Lambda}C_{n}(\widetilde{X})$ is a cycle representing $[X]$, the fundamental class  of $X$,  and $[e]$ is a generator of $C_{n}(\widetilde{X})$.

Write $F^{q}$ for $G_{q}\big(^{\omega}\!\Hom_{\Lambda}(\phantom{A}, \Lambda)\big).$ 

Poncar{\'{e}} duality, together with Lemma 4.3 in \cite{BB2010}, provides the homotopy equivalence of $\Lambda$-modules
\begin{equation*}
  G_{-n+1}\big(\,\underline{\phantom{A}} \smallfrown (1 \otimes [e])\big) \colon F^{n-1}\big(\mathbf{C}(\widetilde{X}) \big) \longrightarrow G_{1}\big(\mathbf{C}(\widetilde{X})\big)
\end{equation*}

Since $\widehat{H}_{1}\big(C(\widetilde{X})\big) =0$,
 \( 
   G_{1}\big(\mathbf{C}(\widetilde{X})\big) = I.
\) 
Thus, 
 \begin{equation*}
    \theta \colon G_{1}\big(\mathbf{C}(\widetilde{X})\big) \longrightarrow I, \quad [c] \longmapsto d_{1}(c)
\end{equation*}
is an isomorphism of $\Lambda$-modules.

Construct the $(n-2)$-type of $X$, $P = P_{n-2}(X)$ by attaching to $X$  cells of dimension $n$ and higher. 

Then the Postnikov section $p \colon X \longrightarrow P$ is the identity on the $(n-1)$-skeleta and for $0 \le i < n$,
\begin{equation*}
    C_{i}(\widetilde{X}) = C_{i}(\widetilde{P})
\end{equation*}

Composing $G_{-n+1}$ with $\theta$, we obtain the homotopy equivalence of left $\Lambda$-modules
\begin{equation*}
   \theta \circ G_{-n+1}\big(\,\underline{\phantom{A}} \smallfrown (1 \otimes [e])\big) \colon F^{n-1}\big(\mathbf{C}(P)\big) \longrightarrow I
\end{equation*}

\subsection{} 
Given a chain complex of free left $\Lambda$-modules, $\mathbf{C}$, we define a homomorphism
\begin{equation*}
    \nu_{\mathbf{C},r} \colon H_{r+1} \big(\mathbb{Z}^{\omega} \otimes_{\Lambda} \mathbf{C}\big) \longrightarrow \big[F^{r}(\mathbf{C}), I \big]
\end{equation*}
with the property that $\nu_{\mathbf{C}(\widetilde{P}), n-1}\left(u_{\ast}\big([X]\big)\right)$ is the homotopy class of the homotopy equivalence 
\[
\theta \circ G_{-n+1}\big(\,\underline{\phantom{A}} \smallfrown (1 \otimes [e])\big) \colon F^{n-1}\big(\mathbf{C}(\widetilde{X}) \big) \longrightarrow I
\]

The short exact sequence of chain complexes
\begin{equation*}
0 \longrightarrow \overline{I}\mathbf{C} \longrightarrow \mathbf{C} \longrightarrow \mathbb{Z}^{\omega}\otimes_{\Lambda} \mathbf{C} \longrightarrow 0
\end{equation*}  
yields the connecting homomorphism 
\begin{align*}
  \delta \colon H_{r+1}\left(\mathbb{Z}^{\omega} \otimes_{\Lambda} \mathbf{C}\right) &\longrightarrow H_{r} (\overline{I}\mathbf{C}) \\ 
  1 \otimes c \ &\longmapsto \ d_{r+1}(c) 
\end{align*}

Then $\nu_{\mathbf{C}, r}$ is the composition of $\delta$ with
\begin{equation*}
   \widehat{\nu}_{\mathbf{C},r} \colon H_{r}(\overline{I}\mathbf{C}) \longrightarrow [F^{r}(\mathbf{C}), I] 
\end{equation*}
where, for $[\lambda .c] \in H_{r}(\overline{I}\mathbf{C})$, $ \widetilde{\nu}_{\mathbf{C},r}\big([\lambda.c]\big)$ is the class of 
\[
    [\varphi] \longmapsto \overline{\varphi(\lambda.c)} = \overline{\varphi(c)} . \overline{\lambda}
\]  

\begin{lem}
  The map $\tilde{\nu}_{\mathbf{C},r}$ is a well defined homomorphism of groups.
\end{lem}
\begin{proof}
   Let $ \lambda .c \in \overline{I}\mathbf{C}$ be a cycle. 
   
   Since $\aug(\overline{\lambda}) = 0$,
   \begin{equation*} 
       \aug( \overline{ \varphi(c) } . \overline{\lambda}) = \overline{\varphi(c)} . \aug(\overline{\lambda}) = 0
   \end{equation*}
   showing that $\overline{\varphi(c)} . \overline{\lambda} \in I$.
   
   Further, for $\varphi^{\prime} = \varphi + d_{r}^{\ast} {\psi} = \varphi + \psi d_{r}$, we obtain 
   \begin{align*}
    \overline{\varphi^{\prime}(\lambda.c)}& = \overline{\varphi(\lambda.c)} + \overline{\psi\big(d_{r}(\lambda.c)\big)} \\ 
    & =  \overline{\varphi(\lambda.c)} + 0 \\
    & = \overline{\varphi(\lambda.c)}
   \end{align*}
   showing that $\overline{\varphi(c)} . \overline{\lambda}$ depends only on $[\varphi]$, not $\varphi$ itself.
   
   Finally, for $\lambda \in \overline{I}, c\in C_{r+1}$, $\tilde{\nu}_{\mathbf{C},r}\big([d_{r+1}(\lambda.c)]\big)$  is represented by the composition of the $\Lambda$-morphisms
   \[
      ^{\omega}\!\Hom_{\Lambda}(C_{r}, \Lambda) \longrightarrow \Lambda, \quad \varphi \longmapsto \overline{\varphi(d_{r+1}c)}
   \]
   followed by
   \[
     \varrho_{\,\overline{\lambda}} \colon \Lambda \longrightarrow I, \quad \mu \longmapsto \mu.\overline{\lambda}
   \]
As $\Lambda$ is free, this composition is null-homotopic, showing that $\tilde{\nu}_{\mathbf{C},r}([\lambda . c])$ is independent of the representative of the class $[\lambda.c]$.
 
 Thus $\widehat{\nu}$ is well defined.
  \end{proof}

The \emph{Turaev map} is the morphism $\nu_{\mathbf{C}(\widetilde{P}_{n-2}), n-1}$.

\begin{lem}\label{lem:nu}
 $\nu_{\mathbf{C}(\widetilde{P}), n-1}\big( p_{\ast}([X])\big) = \big[\theta \circ G_{-n+1}\big(\,\underline{\phantom{A}} \smallfrown (1\otimes[e]) \big)\big]$
\end{lem}

\begin{proof}
  Take a diagonal $\Delta \colon \mathbf{C}(\widetilde{X}) \longrightarrow  \mathbf{C}(\widetilde{X}) \otimes_{\mathbb{Z}}  \mathbf{C}(\widetilde{X})$ and a chain homotopy $\alpha \colon  \mathbf{C}(\widetilde{X}) \longrightarrow  \mathbf{C}(\widetilde{X}) $ such that 
  \begin{equation*}
     id - (id \otimes \aug)\Delta = d \alpha + \alpha d
  \end{equation*}
  where we identify $\mathbf{C} \otimes_{\mathbb{Z}} \mathbb{Z}$ with $ \mathbf{C}$.
  
  Suppose that
  \begin{equation*}
      \Delta e = e \otimes \lambda + \sum_{\ell}\sum_{0 \le   i   < n}x_{\ell,i} \otimes y_{\ell, n-i}
  \end{equation*}
  Then 
  \begin{align*}
     e & = (id \otimes \aug)\Big(  e \otimes \lambda + 
     \sum_{\ell}\sum_{0 \le i < n} 
     x_{\ell,i} \otimes y_{\ell, n-i}\Big) + d\alpha e + \alpha de \\ 
      &= (id \otimes \aug)\Big(  e \otimes \lambda + 
     \sum_{\ell}\sum_{0 \le i < n} 
     x_{\ell,i} \otimes y_{\ell, n-i}\Big) + 0 + \alpha de &&\text{as }C_{n+1}(\widetilde{X})=0 \\ 
      & = e \otimes \aug(\lambda) + \alpha de \\
      & = \aug(\lambda) e + \alpha de
  \end{align*}
  where we have identified $C_{n}(\widetilde{X})\otimes C_{0}(\widetilde{X}) = C_{n}(\widetilde{X}) \otimes \Lambda$ with $C_{n}(\widetilde{X})$, using $c \otimes \lambda \longmapsto \lambda . e$.
  
Since  $[1 \otimes e]$ generates $\widehat{H}_{n}\left(X; \mathbb{Z}^{\omega}\right) \cong \mathbb{Z}$,
  \begin{align*}
      1 \otimes e & = 1 \otimes \aug(\lambda .e) + 1 \otimes \alpha de \\ 
                           & = \aug(\lambda) \otimes e + (1 \otimes \alpha) (1 \otimes d)( 1 \otimes e) \\ 
                           & = \aug(\lambda) \otimes e +(1 \otimes \alpha)(0) 
\end{align*}
whence $\aug(\lambda) =1$, or $\aug( \lambda - 1) =0$.

Hence, $\lambda -1 \in I = \im(d_{2})$, or, $ \lambda = 1 + d_{2}(c_{2})$ for some $c_{2} \in C_{2}(\widetilde{X})$. 
  
  Thus, given $\varphi \in \Hom_{\Lambda}\big(C_{n}(\widetilde{X}), \Lambda\big)$, 
  \begin{align*}
      \varphi \smallfrown (1 \otimes e) & = \varphi(e) \otimes \lambda\\
             & = \varphi(e) \otimes \big(1 + d_{2}(c_{2})\big) \\ 
             & = \overline{\varphi(e)}\big(1 + d_{2}(c_{2})\big)
  \end{align*}
  Since  $\mathbf{G}\big(\,\underline{\phantom{A}} \smallfrown (1 \otimes [e])\big)$ is a chain map,
    \begin{align*}
      \Big(\theta \circ G_{-n+1}\big(\,\underline{\phantom{A}} \smallfrown (1 \otimes [e])\big)\Big)\big([\varphi]\big) & = d_{1}\Big(G_{-n+1}\big(\,\underline{\phantom{A}} \smallfrown (1 \otimes [e])\big)\big([\varphi]\big) \Big) \\
            & = G_{-n}\big(\,\underline{\phantom{A}} \smallfrown  (1 \otimes [e])\big)\Big(\big[d_{n}^{\ast}(\varphi)\big]\Big) 
            \\ 
            & = d_{n}^{\ast}(\varphi) \smallfrown (1 \otimes [e]) \\ 
            & = \overline{d_{n}^{\ast}(\varphi)(e)}(1 + dc_{2}d_{2}) \\ 
            & = \overline{\varphi\big(d_{n}(e)\big)}(1 + c_{2}d_{2})
  \end{align*}
  On the other hand,
  \begin{align*}
     \nu_{\mathbf{C}(\widetilde{P}), n-1}\big(p_{\ast}([X])\big)\big([\varphi]\big) & = \nu_{\mathbf{C}(\widetilde{P}), n-1}\big( p_{\ast}(1 \otimes [e]) \big)\big([\varphi]\big) \\ 
         & = \widehat{\nu}_{\mathbf{C}(\widetilde{P}), n-1}\big(\delta (1 \otimes [e]) \big)\big([\varphi]\big) \\ 
         & = \widehat{\nu}_{\mathbf{C}(\widetilde{P}), n-1}\big( [d_{n}(e)] \big)\big([\varphi]\big)
  \end{align*}
  Hence, by definition, $\widehat{\nu}_{\mathbf{C}(\widetilde{P}), n-1}\big(p_{\ast}([X])\big)$ is represented by the $\Lambda$-morphism
  \begin{equation*}
  F^{n-1}\big(\mathbf{C}(\widetilde{P})\big) \longrightarrow I, \quad [\varphi] \longmapsto \overline{\varphi\big(d_{n}(e)\big)}
  \end{equation*}
  To conclude the proof, note that
  \begin{equation*}
     F^{n-1}\mathbf{C}(\widetilde{P}) \longrightarrow I, \quad [\varphi] \longmapsto \overline{\varphi\big(d_{n}(e)\big)}.d_{2}(c_{2})
  \end{equation*}
  is null-homotopic.
\end{proof}

\subsection{}
As $\theta \circ G_{-n+1} \big(\,\underline{\phantom{A}} \smallfrown (1  \otimes [e])\big)$ is a homotopy equivalence of $\Lambda$-modules, Lemma \ref{lem:nu} provides a necessary condition for realisation.

\begin{thm}\label{thm:necforrealisation}
 If a fundamental triple $(P_{n-2}, \omega, \mu)$, consisting of an $(n-2)$-type, $P_{n-2}$, $\omega \in H^{1}\big(P_{n-2}; \raisebox{0.3ex}{$\mathbb{Z}$}/\raisebox{-0.2ex}{$2\mathbb{Z}$}\big)$ and $\mu \in H_{n}(P_{n-2}; \mathbb{Z}^{\omega})$, is the fundamental triple of a PD$^{n}$-complex $X$, then $\nu_{\mathbf{C}(\widetilde{P}_{n-2}), n-1}(\mu)$,   the Turaev map applied to $\mu$, is a homotopy equivalence of modules over $\Lambda = \mathbb{Z}[\pi_{1}(P_{n-2})]$.
\end{thm}
\begin{proof}
   Let $P_{n-2}$ be an $(n-2)$-type. Take $\omega \in H^{1}\big(P_{n-2} ;\raisebox{0.3ex}{$\mathbb{Z}$}/\raisebox{-0.2ex}{$2\mathbb{Z}$}\big)$ and $\mu \in H_{n}(P_{n-2}; \mathbb{Z}^{\omega})$. Suppose that $(P_{n-2}, \omega, \mu)$ is the fundamental triple of the PD$^{n}$-complex, $X$. 
   
   If $P^{\prime}_{n-2}$ is an $(n-2)$-type obtained by attaching to $X$ cells of dimension $n$ and higher, then there is  a homotopy equivalence $f \colon P_{n-2} \longrightarrow P^{\prime}_{n-2}$ with $f_{\ast}(\mu) = i_{\ast}([X])$, where $i \colon X \longrightarrow P^{\prime}_{n-2}$ is the inclusion.
   
   By Lemma \ref{lem:nu},
 \begin{equation*}
    \nu_{\mathbf{C}(\widetilde{P^{\prime}}_{n-2}) n-1}\big(i_{\ast}[X]\big) =\big[\theta \circ G_{-n+1}\big(\,\underline{\phantom{A}} \smallfrown (1  \otimes [e])\big)\big]
 \end{equation*}
and is hence a homotopy equivalence of $\Lambda$-modules. 

The diagram
\begin{equation*}
  \begin{CD}
    F^{n-1}\big(\mathbf{C}(\widetilde{P}_{n-2})\big) @>{\quad \qquad \qquad \nu_{ \mathbf{C}(\widetilde{P}_{n-2}), n-1}(\mu) \qquad \qquad \quad  }>> I \\
@V{F^{n-1}(f_{\ast})}VV    @| \\ 
 F^{n-1}\big(\mathbf{C}(\widetilde{P}^{\prime}_{n-2})\big) @>>{\ \nu_{ \mathbf{C}(\widetilde{P^{\prime}}_{n-2})n-1}\big(f_{\ast}(\mu)\big) =\nu_{ \mathbf{C}(\widetilde{P}^{\prime}_{n-2}), n-1}(i_{\ast}[X])  \ }> I \\
  \end{CD}
\end{equation*}
commutes, so that $\nu_{ \mathbf{C}(\widetilde{P}_{n-2}), n-1}(\mu)$ is a homotopy equivalence of modules.
\end{proof}

Necessary conditions for realisation are a corollary to Lemma \ref{lem:Hi=0} and Theorem \ref{thm:necforrealisation}

\begin{cor}[\textbf{Conditions for Realisability}]\label{cor:ftem}
Let  $(G, \omega, \mu)$ be a triple with $G$ a group, $\omega \in H^{1}\big(G; \raisebox{0.3ex}{$\mathbb{Z}$}/\raisebox{-0.2ex}{$2 \mathbb{Z}$}\big)$ and $\mu \in H_{n}(G; \mathbb{Z}^{\omega})$.
 
If  $(G, \omega, \mu)$ is the fundamental triple of a PD$^{n}$-complex, $X$, with $(n-2)$-connected universal cover, then $\nu_{\mathbf{C}(\widetilde{K}), n-1}(\mu)$ is a homotopy equivalence of $\Lambda = \mathbb{Z}[\pi_{1}(X)]$ modules, and for $1 < i < n-1 $
 \[
   H^{i}(G; \,^{\omega}\!\Lambda ) = 0 
 \]
\end{cor}

\section{Sufficiency of the Realisation Condition}\label{sect:sufficiency}

Let $G$ be a finitely presentable group of type FP$_{n}$, with $H^{i}(G; \,^{\omega}\! \mathbb{Z}[G]) =0$ for $1 < i < n-1$, and  $n \ge 3$.


Let $K^{\prime} = K(G,1)$ be an Eilenberg-Mac Lane space and identify the (co-)homologies of $G$ and $K^{\prime}$.

Given $\omega \in H^{1}\big(G; \raisebox{0.3ex}{$\mathbb{Z}$}/\raisebox{-0.2ex}{$2\mathbb{Z}$} \big)$ and $\mu \in H_{n}\big( G; \mathbb{Z}^{\omega}\big) $, with $\nu_{\mathbf{C}(\widetilde{K^{\prime}}), n-1}(\mu) $ a class of homotopy equivalences of $\mathbb{Z}[G]$-modules, we construct a PD$^{n}$-complex, $X$, with $(n-2)$-connected universal cover and fundamental triple $(G, \omega, \mu)$.

Let $\widetilde{K^{\prime}} \longrightarrow K^{\prime}$ be the universal covering space of $K^{\prime}$.

By our assumptions about $G$, we can choose ${K^{\prime}}$ with finitely many cells in each dimension. 

Let $h \colon F^{n-1}\big( \mathbf{C}(\widetilde{K^{\prime}})\big) \longrightarrow I$ be a representative of  $\nu_{\mathbf{C}(\widetilde{K^{\prime}}), n-1}(\mu)$. Then $h$ is a homotopy equivalence of $\Lambda$-modules. By Theorem 4.1 and Observation 1 in \cite{BB2010}, $h$ factors as
\begin{equation*}
 \xymatrix{
    F^{n-1}\big(\mathbf{C}(\widetilde{K^{\prime}})\big) \  \ar@{>->}[r] &  \  F^{n-1}\big(\mathbf{C}(\widetilde{K^{\prime}})\big) \oplus \Lambda^{m} \, \  \ar@{>->}[r]^{\qquad \quad  j \ } &  \  I \oplus P  \   \ar@{->>}[r] &  \   I
 }
 \end{equation*}
for some $m \in \mathbb{N}$ and $P$ a projective $\Lambda$-module.

Let $B = (e^{0} \cup e^{n-1}) \cup e^{n}$ be the $n$-dimensional ball and replace $K^{\prime}$ by the Eilenberg-Mac Lane space $K = K^{\prime} \cup \Big({\displaystyle \bigcup_{i=1}^{m} B} \Big)$. 

Then $F^{n-1}\big(
                                  \mathbf{C}(\widetilde{K})
                           \big) = F^{n-1}\big(\mathbf{C}(\widetilde{K^{\prime}})\big) \oplus \Lambda^{m}$ and the factorisation of $h$ becomes
\begin{equation*}
  \xymatrix{
       h \colon F^{n-1}\big(\mathbf{C}(\widetilde{K})\big)\quad \ar@{>->>}[r]^{\qquad \qquad \qquad j \qquad\qquad} & \ I \oplus P \ \ar@{->>}[r]^{\qquad \qquad \qquad pr_{I} \qquad\qquad \quad} &\ I
  }
\end{equation*}

Consider the $\Lambda$-morphism, $\varphi$, given by the composition$$  \xymatrix{
    C^{n-1}(\widetilde{K}) = \,  ^{\omega}\!\Hom_{\Lambda}\big(C_{n-1}(\widetilde{K}), \Lambda \big)  \ar@{->>}[r]^{ \qquad \qquad \pi \ } & F^{n-1}\big(\mathbf{C}(\widetilde{K})\big) \ \ar@{>->}[r]^{\qquad j \quad} &    I \oplus P \ \ar@{->}[r]^{\phantom{a}\qquad\qquad \qquad{\ \text{\scriptsize $\begin{bmatrix} i & 0 \\ 0 & id \end{bmatrix} $} \ }\qquad \qquad\phantom{a}\qquad} & \ \Lambda \oplus P
  }
$$

where $\pi$ is the natural projection onto the cokernel, and $ i \colon I \rightarrowtail \Lambda$ is the inclusion.

Since $F^{n-1}\big(\mathbf{C}(\widetilde{K}) \big) = \raisebox{0.3ex}{$C_{n-1}(\widetilde{K})$}/\raisebox{-0.2ex}{$\im(d^{\ast}_{n-1})$}$ by definition, 
$\varphi \circ d^{\ast}_{n-1}=0$. 

The dual of $\varphi$ is a $\Lambda$-morphism
\begin{equation*}
    \varphi^{\ast} \colon (\Lambda \oplus P )^{\ast} \longrightarrow C_{n-1}(\widetilde{K})^{\ast\ast} 
    \end{equation*}
    As $C_{n-1}(\widetilde{K}^{[n-1]})$ is a finitely generated free $\Lambda$-module, the 
 natural   map 
 \[
    ^{\omega}\varepsilon \colon C_{n-1}(\widetilde{K}^{[n-1]}) \longrightarrow C_{n-1}(\widetilde{K}^{[n-1]})^{\ast\ast}
 \] is an isomorphism and we define
    \begin{equation*}
       d_{n} \colon (\Lambda \oplus P)^{\ast} \longrightarrow C_{n-1}(\widetilde{K})
    \end{equation*}
    by
    \[
       d_{n} = \, \big(^{\omega}\varepsilon\big)^{-1} \circ \,\varphi^{\ast}
    \]
    By the naturality of \,$^{\omega}\varepsilon$,
    \begin{align*}
         d_{n-1} \circ d_{n} & = (^{\omega}\varepsilon)^{-1} \circ d_{n-1}^{\ast\ast} \circ ^{\omega}\varepsilon \circ (^{\omega}\varepsilon)^{-1} \circ \varphi^{\ast} \\
                                           & = (^{\omega}\varepsilon)^{-1} \circ (\varphi \circ d_{n-1}^{\ast})^{\ast} \\ 
                                           &=0
    \end{align*}

 We first consider the case when $P$ is free, so that $P \cong \Lambda^{q}$ for some $q \in \mathbb{N}$ and $\Lambda \oplus P \cong \Lambda^{q+1}$.
 
 \medskip

Since  ${\widetilde{K}}^{[n-1]}$ is $(n-2)$-connected, the Hurewicz homomomorphism
\[
    h_{q} \colon \pi_{q}\big(\widetilde{K}^{[n-1]}\big) \longrightarrow H_{q}\big(\widetilde{K}^{[n-1]}\big)
\] is an isomorphism for $ q \le n-1$ and we obtain the map
\begin{align*}
      \varphi^{\prime}\colon \Lambda^{q+1} \cong (\Lambda \oplus P)^{\ast}  &\longrightarrow \ker(d_{n-1}) = H_{n-1}(\widetilde{K}^{[n-1]}) \xrightarrow[]{\ h_{n-1}^{-1} \ } \pi_{n-1}(\widetilde{K}^{[n-1]})\\ 
      x \quad & \longmapsto \quad h_{n-1}^{-1} \big( [d_{n}(x)] \big)
\end{align*}

Let $\mathbf{C} $ be the chain complex of $\Lambda$-modules
\begin{equation*}
 \Lambda^{q+1} \cong \big(\Lambda \oplus P\big)^{\ast} \xrightarrow{\ d_{n} \ } C_{n-1}(\widetilde{K}^{[n-1]}) \xrightarrow{ \ d_{n-1} \ } \ \cdots \ \xrightarrow{ \ \quad \ } C_{1}(\widetilde{K}^{[n-1]}) \xrightarrow{ \ \quad \ } \Lambda
\end{equation*}

Then  $Y = ( \mathbf{C}, \overline{\varphi}, K^{[n-1]})$ is a homotopy system of order $n$.

As $C_{i} = 0$ for $i > n$, \ \( \widehat{H}^{n+2}(Y ; \Gamma_{n}Y) = 0$ and, by Proposition 8.3 in \cite{HJB-BB2010}, there is a homotopy system $(\mathbf{C}, 0, X)$ of order $n+1$ realising $Y$, with $X$ an $n$-dimensional $CW$-complex.

By construction, $\mathbf{C}(\widetilde{X}) = \mathbf{C}$ and the universal cover of $X$ is $(n-2)$-connected.

The inclusion $ i \colon K^{[n-1]} \longrightarrow K$ extends to a map 
\begin{equation*}
    f  \colon X \longrightarrow K = K(G,1)
\end{equation*}
and we may consider $\omega \in H^{1}\!\left( K; \raisebox{0.3ex}{$\mathbb{Z}$}/\raisebox{-0.2ex}{$2\mathbb{Z}$}\right)$ as element of $ H^{1}\!\left( X; \raisebox{0.3ex}{$\mathbb{Z}$}/\raisebox{-0.2ex}{$2\mathbb{Z}$}\right)$.

\begin{prop}\label{prop:realisefree}
 $X$ is a PD$^{n}$-complex with fundamental triple $(G, \omega, \mu)$, that is
 \begin{itemize}
   \item[(i)] $ \mathbb{Z} \cong H_{n}(X; \mathbb{Z}^{\omega}) = \langle [X] \rangle $;
   
   \item[(ii)] $f_{\ast}([X]) = \mu$;
   
   \item[(iii)] $\underline{\phantom{A}} \smallfrown [X] \colon H^{r}(X; \, ^{\omega}\!\Lambda) \longrightarrow H_{n-r}(X ; \Lambda)$ is an isomorphism for every $r \in \mathbb{Z}$.
 \end{itemize}
\end{prop} 

\begin{proof}

\textbf{(i)} \ As $\mathbf{C}(\widetilde{X}) = \mathbf{C}$ is a chain complex of finitely generated free $\Lambda$-modules, the natural map
 \begin{align*}
     \eta_{\mathbf{C}} \colon \mathbb{Z}^{\omega} \otimes_{\Lambda} \mathbf{C} \longrightarrow \Hom_{\Lambda}\big(\,^{\omega}\!\Hom_{\Lambda}(\mathbf{C}, \Lambda), \mathbb{Z}\big)
 \end{align*}
 is an isomorphism.
 
 Hence, \( H_{n} \big(X ; \mathbb{Z}^{\omega} \big) = \ker(1 \otimes d_{n}) \cong \ker (\varphi^{+}) \), for 
 \begin{align*}
   \varphi^{+} \colon \Hom_{\Lambda} \big( 
   \Lambda \oplus \Lambda^{q}
   , \mathbb{Z} \big) &\longrightarrow \Hom_{\Lambda}\big( ^{\omega}\!\Hom_{\Lambda}( C_{n-1}(\widetilde{K}^{[n-1]}), \Lambda), \mathbb{Z}\big),\\ 
   \psi \qquad  &\longmapsto \qquad  \psi \circ \varphi
 \end{align*}
 Since, by definition,
 \begin{equation*}
      \varphi := \begin{bmatrix} i & 0 \\ 0 & id \end{bmatrix} 
      \circ j \circ \pi \colon {\color{red}^{\omega}} \Hom_{\Lambda}\big(C_{n-1}(\widetilde{K}) , \Lambda \big) \longrightarrow \Lambda \oplus \Lambda^{q}
 \end{equation*}
we obtain
\begin{equation*}
    \varphi^{+} = \pi^{+} \circ j^{+} \circ  \begin{bmatrix} 
       i & 0 \\
       0 & id 
       \end{bmatrix}^{+}
\end{equation*}

Since both $\pi$ and $j$ are surjective, both $\pi^{+}$ and $j^{+}$ are injective, whence
\begin{align*}
    \ker(\varphi^{+}) & = \ker \left( \begin{bmatrix} 
       i & 0 \\
       0 & id 
       \end{bmatrix}^{+}\right) \\[1ex] 
       & = \ker \left( \begin{bmatrix} 
       i^{+} & 0 \\
       0 & id 
       \end{bmatrix}\right)\\[1ex] 
       & \cong \ker(i^{+})
  \end{align*}
  
  But $I$ is generated by elements $1 - g \ \ (g\in G)$, and, for $\psi \in \Hom_{\Lambda}( \Lambda, \mathbb{Z})$, we have 
  \begin{align*}
      (\psi \circ i)(1-g) &= \psi(1 - g) \\
                                    &= \psi( 1- g.1) \\
                                    & = \psi(1)- g.\psi(1)\\ 
                                    & = 0
  \end{align*}
  Hence, 
  \begin{equation*}
      \ker (\varphi^{+}) \cong \Hom_{\Lambda}( \Lambda, \mathbb{Z}) \cong \mathbb{Z}
  \end{equation*}
  generated by 
  \[
     \aug \circ \, pr_{\Lambda} \colon \Lambda \oplus \Lambda^{q} \longrightarrow \mathbb{Z}
  \]the projection onto the first factor, followed by the augmentation.
  
  Let $[X] = [1 \otimes x] \in H_{n}(X;\mathbb{Z}^{\omega})$ be the homology class corresponding to $ \aug \circ \, pr_{\Lambda}$ under the isomorphism  \( H_{n} \big(X ; \mathbb{Z}^{\omega} \big) = \ker(1 \otimes d_{n}) \cong \ker (\varphi^{+}) \cong \Hom_{\Lambda}( \Lambda, \mathbb{Z})\).
  
  Then $x \in (\Lambda \oplus \Lambda^{q})^{\ast}$ is projection onto the first factor.
  
    \smallskip
  
  \textbf{(ii)\phantom{i}} 
  
  By the proof of Lemma \ref{lem:nu}, $\nu_{\mathbf{C}(\widetilde{X}), n-1}\big([X]\big)$ is represented by
  \begin{equation*}
     F^{n-1} \big(\mathbf{C}(\widetilde{X})\big) \longrightarrow I, \quad [\psi] \longmapsto \overline{\psi(d_{n}(x))}
  \end{equation*}

 Thus, given $\psi \in C_{n-1}(\widetilde{X})^{\ast} = C_{n-1}(\widetilde{K})^{\ast}$,
  \begin{align*}
   \overline{\psi(d_{n}(x))} & = \overline{\psi\big(\,^{\omega}\!\ev^{-1}(\varphi^{\ast}(x))\big)} \\ 
                    & = \overline{\psi\big(\,^{\omega}\!\ev^{-1}(x \circ \varphi)\big)} \\ 
                    & = \,^{\omega}\!\ev \big(\,^{\omega}\!\ev^{-1}(x \circ \varphi)\big)(\psi) \\ 
                    & =( x \circ \varphi)(\psi) \\
                    & = (x \circ \begin{bmatrix} i & 0 \\ 0 & id\end{bmatrix} \circ j \circ \pi)(\psi) \\ 
                    & = (x \circ \begin{bmatrix} i & 0 \\ 0 & id\end{bmatrix} \circ j)([\psi]) \\ 
                    & = (i \circ pr_{I} \circ j)([\psi]))\\
                    & = h([\psi])
  \end{align*}

  Hence, $\nu_{\mathbf{C}(\widetilde{X}), n-1}([X])$ is the homotopy class of $h$ so that 
  \begin{align*}
     \nu_{\mathbf{C}(\widetilde{K}), n-1}(\mu) & = \nu_{\mathbf{C}(\widetilde{X}), n-1}([X]) \\ 
                   & = \nu_{\mathbf{C}(\widetilde{K}), n-1}(f_{\ast}([X]))
  \end{align*}
By Lemma 2.5 in \cite{VGT1990}, $\nu_{\mathbf{C}(\widetilde{K}), n-1}$ is injective, whence $\mu = f_{\ast}([X])$.

\smallskip
\textbf{(iii)} \ Given $1 \le i < n-1$
\begin{equation*}
    H_{i} (\widetilde{X}; \Lambda) = H_{i}({\widetilde{K}}^{[n-1]}; \Lambda) = 0
\end{equation*}
  and, by the definition of $\varphi$,
  \begin{equation*}
      H^{n-1}(\widetilde{X};  \, ^{\omega}\!\Lambda^{\omega}) = 0
  \end{equation*}
  Moreover,
  \begin{align*}
    H^{i}(\widetilde{X}; \, \Lambda^{\omega}) & = H^{i}({\widetilde{K}^{[n-1]}}; \,\Lambda^{\omega}) \\ 
      & \cong H^{i}(G;\Lambda^{\omega}) \\
      & = 0 &&\text{by hypothesis}
  \end{align*}
  Thus
  \[
    \underline{\phantom{A}} \smallfrown (1 \otimes [X]) \colon H^{i}(X; \,^{\omega}\!\Lambda) \longrightarrow H_{n-i}(X ; \Lambda)
  \]
  is an isomorphism for $0 < i < n$, and it remains only to consider $i =0 ,n$.
  
  As $P$ and $\Lambda \oplus P$ are free, $\mathbf{C}(X)$ is a chain complex of free $\Lambda$-modules. Since the (twisted) evaluation map from a finitely generated free $\Lambda$-module to its double dual is an isomorphism, we obtain
  \begin{align*}
     H^{n}(\widetilde{X}; \, ^{\omega}\!\Lambda) & = H_{-n}\big(\,^{\omega}\!\Hom_{\Lambda}(\mathbf{C}(X), \, ^{\omega}\!\Lambda)) \\[1ex]
           & = \raisebox{0.5ex}{$^{\omega}\!\Hom_{\lambda}(C_{n}(X), \, ^{\omega}\!\Lambda)$}\!/\raisebox{-0.3ex}{$\im(\varphi^{\ast})^{\ast}$} \\[1ex]
             & = \raisebox{0.5ex}{$(\Lambda \oplus P)^{\ast\ast}$}\!/\raisebox{-0.3ex}{$\im(\varphi^{\ast})^{\ast}$} \\[1ex]
           & \cong \raisebox{0.5ex}{$\Lambda \oplus P$}\!/\raisebox{-0.3ex}{$\im(\varphi)$} \\
           & \cong \raisebox{0.5ex}{$\Lambda$}/\!\raisebox{-0.3ex}{$I$} \\
           & \cong \mathbb{Z}
  \end{align*}
  
  The class, $[\gamma]$, of the image of $(1,0) \in \Lambda \oplus P$ under the (twisted) evaluation isomorphism generates $H^{n}(X; \, ^{\omega}\!\Lambda)$ and so, by Lemma 4.4 of \cite{BB2010},
  \begin{align*}
      [\gamma] \smallfrown [X] & = [\gamma] \smallfrown [1 \otimes x] \\ 
        & = [\overline{\gamma(x)}. e_{0}] \\ 
        & = [e_{0}] &&\text{as $\gamma(1 \otimes x) = 1$}
  \end{align*}
  where $e_{0} \in C_{0}(X)$ is a chain representing the base point.
  
 Thus, 
 \[
    \underline{\phantom{A}} \smallfrown [X] \colon H^{n}(\widetilde{X} ; \, ^{\omega}\!\Lambda) \longrightarrow H_{0}(\widetilde{X}; \Lambda)
 \]
 is an isomorphism and $\underline{\phantom{A}} \smallfrown (1 \otimes x)$ yields the chain homotopy equivalence
 \begin{equation*}
    \begin{split}
       \begin{CD}
          0 @>>> \im d_{n-1}^{\ast} @>>>  C_{n-1}(X)^{\ast} @>>> C_{n}(X)^{\ast} @>>> 0 \\
          @VVV @VVV @VV{\underline{\phantom{A}} \smallfrown (1 \otimes x)}V @VV{\underline{\phantom{A}} \smallfrown (1 \otimes x)}V @VVV \\
          0 @>>> \im d_{2} @>>>C_{1}(X) @>>> C_{0}(X) @>>> 0
       \end{CD}
    \end{split}
 \end{equation*}
 Applying the functor $^{\omega}\!\Hom_{\Lambda}(\underline{\phantom{A}}\, , \Lambda)$, we obtain a chain homotopy induced by $\big(\underline{\phantom{A}} \smallfrown (1\otimes x)\big)^{\ast}$. Hence, $\big(\underline{\phantom{A}} \smallfrown (1\otimes x)\big)^{\ast}$ induces the{\color{red}{?}} isomorphism
  \[
              H^{0}(X; \Lambda) \longrightarrow H_{n}(X ; \Lambda)
 \]
By Lemma 2.1 in \cite{BB2010}, $\big(\, \underline{\phantom{A}} \smallfrown (1 \otimes x)\big)^{\ast}$ induces an isomorphism in homology if and only if $\underline{\phantom{A}} \smallfrown (1 \otimes x)$ does, whence
\[
\underline{\phantom{A}} \smallfrown [X] \colon H^{0}(X; \, ^{\omega}\Lambda) \longrightarrow H_{n}(X; \Lambda)
\]
is an isomorphism.
\end{proof}

 Suppose now that $P$ is projective, but not free.
 
Then there is a finitely generated $\Lambda$-module $Q$ and a natural number $q$ with $P^{\ast} \oplus Q \cong \Lambda^{q}$.  The natural isomorphisms
 \begin{align*}
  ( \Lambda \oplus P)^{\ast} \oplus \Lambda^{\infty} & \cong  \Lambda^{\ast} \oplus P^{\ast} \oplus(Q \oplus P^{\ast} \oplus \cdots) \\ 
  & \cong \Lambda \oplus (P^{\ast}  \oplus Q \oplus P^{\ast} \oplus Q \cdots) \\ 
  & \cong\Lambda^{\infty} 
 \end{align*}
show that $( \Lambda \oplus P)^{\ast} \oplus \Lambda^{\infty}$ is a free $\Lambda$-module. 
 
 Consider the chain complex $\mathbf{D}$
 \begin{align*}
 0\longrightarrow &( \Lambda \oplus P)^{\ast} \oplus \Lambda^{\infty} \xrightarrow[]{\ \begin{bmatrix} d_{n}  & 0 \\ 0 & id \end{bmatrix} \ } C_{n-1}\Big({\widetilde{K}}^{[n-1]}\Big)\oplus \Lambda^{\infty} 
 \xrightarrow{\ [ d_{n-1} \ 0 ]  \ } \\[1ex] 
 & C_{n-2}\Big({\widetilde{K}}^{[n-1]} \Big)\xrightarrow{\ d_{n-2}\ } C_{n-3}\Big({\widetilde{K}}^{[n-1]} \Big) \longrightarrow \cdots
 \end{align*}

We may attach infinitely many $n$-balls to ${K}^{[n-1]}$ to obtain a $CW$-complex, $K^{\prime}$, whose cellular chain complex coincides with the chain complex $\mathbf{D}$ in dimensions $\le n-1$.

Then  ${\widetilde{K'}}^{[n-1]}$  is $(n-2)$-connected, and the Hurewicz homomomorphism
\[
    h_{q} \colon \pi_{q}\big(\widetilde{K'}^{[n-1]}\big) \longrightarrow H_{q}\big(\widetilde{K'}^{[n-1]}\big)
\] is an isomorphism for $ q \le n-1$. Defining the map
\begin{align*}
      \varphi^{\prime} \colon ( \Lambda \oplus P)^{\ast} \oplus \Lambda^{\infty} &\longrightarrow \ker(d_{n-1}) = H_{n-1}(\widetilde{K'}^{[n-1]}) \xrightarrow[]{\ h_{n-1}^{-1} \ } \pi_{n-1}(\widetilde{K'}^{[n-1]})\\ 
      x \quad & \longmapsto \quad h_{n-1}^{-1} \big( [d_{n}(x)] \big)
\end{align*}
we obtain the homotopy system $Y' = ( \mathbf{D}, \check{\varphi}, K'^{[n-1]})$ of order $n$.

As $D_{i} = 0$ for $i > n$, \ \( \widehat{H}^{n+2}(Y' ; \Gamma_{n}Y') = 0$. 

By Proposition 8.3 in \cite{HJB-BB2010}, there is then a homotopy system $(\mathbf{C}, 0, X')$ of order $n+1$ realising $Y'$, with $X'$ an $n$-dimensional $CW$-complex.
 
Observe that $\mathbf{D}$, the chain complex of $X$, is chain homotopy equivalent to the chain complex $\mathbf{W}$
 \begin{align*}
 \cdots \longrightarrow &P^{\ast} \oplus Q \xrightarrow[]{\ \begin{bmatrix} id & 0 \\ 0 & 0 \end{bmatrix} \ } P^{\ast} \oplus Q \xrightarrow[]{\begin{bmatrix} 0 & 0 \\ 0 & id \end{bmatrix} \ } P^{\ast} \oplus Q \xrightarrow[] {\begin{bmatrix} 0 & 0 \\  0 & 0 \\ 0 & id \end{bmatrix} \ } \\[2ex]
     &(\Lambda \oplus P)^{\ast}\oplus Q \xrightarrow[]{\begin{bmatrix} d_{n} \\ 0 \end{bmatrix}}  C_{n-1}({\widehat{K}}^{[n-1]}) \xrightarrow[]{\ d_{n-1} \ } \\[2ex]
     & C_{n-2}({\widehat{K}}^{[n-1]})  \xrightarrow{\ d_{n-2}\ } C_{n-3}({\widetilde{K}}^{[n-1]} ) \longrightarrow \cdots
 \end{align*}
 
 By  Theorem 2 of \cite{W1966}, there is a $CW$-complex, $X$, with cellular chain complex $\mathbf{W}$, homotopy equivalent to $X^{\prime}$. 
 
  The proof that $X$ realises $(G, \omega, \mu)$ is similar to the proof of Proposition \ref{prop:realisefree}.  
 
 This completes the proof of Theorem A.

 \section{Decomposition as Connected Sum}\label{sect:connsum}
 
Wall constructed a new PD$^{n}$ complex from given ones  using the \emph{connected sum of PD$^{n}$ complexes} (cf \cite{W1967}). This allows  PD$^{n}$ complexes to be decomposed as connected sum of other, simpler PD$^{n}$ complexes.

Take PD$^{n}$ complexes $(X_{k}, \omega_{k}, [X_{k}])$ for $k = 1, 2$.

Then we may express $X_{k}$ as the mapping cone
\begin{equation*}
       X_{k} = X_{k}^{\prime}\cup_{f_{k}}e_{k}^{n}
\end{equation*}
for suitable $f_{k} \colon S^{n-1} \longrightarrow X_{k}^{\prime}$.

Here, $X_{k}^{\prime}$ is an $(n-1)$-dimensional $CW$-complex when $n > 3$, and when $n=3$, $X_{k}^{\prime}$ is 3-dimensional with $H^{3}(X_{k}^{\prime}; B) = 0$ for all coefficient modules, $B$.

For $k = 1, 2$, let  $in_{k} \colon X_{k}^{\prime} \longrightarrow X_{1}^{\prime} \vee X_{2}^{\prime}$ be the canonical inclusion of the $k^{\text{th}}$ summand and put
\begin{equation*}
     \widehat{f}_{k} :=\iota_{k} \circ f_{k} \colon S^{n-1} \longrightarrow X_{1}^{\prime} \vee X_{2}^{\prime}
\end{equation*} 
so that $\widehat{f}_{k}$ determines an element of $\pi_{n-1}\big(X_{1}^{\prime} \vee X_{2}^{\prime}\big)$

Let  $f_{1} + f_{2} \colon S^{n-1} \longrightarrow X_{1}^{\prime} \vee X_{2}^{\prime}$ represent the homotopy class $[\widehat{f}_{1}] + [\widehat{f}_{2}]$. 

Then the connected sum,  $X = X_{1} \# X_{2} $, of $X_{1}$ and $X_{2}$ is  the mapping cone of $f_{1} + f_{2}$,
\begin{equation*}
     X_{1} \# X_{2} := \big(X_{1}^{\prime} \vee X_{2}^{\prime}\big) \cup_{f_{1}+f_{2}} e^{n}
\end{equation*}

It follows from the Seifert-van Kampen Theorem that 
\begin{equation*}
    \pi_{1}(X) = \pi_{1}(X_{1}) \ast \pi_{1}(X_{2})
\end{equation*}

The canonical inclusion $in_{k} \colon \pi_{1}(X_{k}) \longrightarrow \pi_{1}(X)$ induces a (left, resp.~right) $\mathbb{Z}[\pi_{1}(X_{k})]$-module structure on any  (left resp.~right) $\Lambda = \mathbb{Z}[\pi_{1}(X)]$-module. In particular $\Lambda$ is a $\pi_{1}(X_{k})$-bimodule.

By the universal property of the free product, the morphisms $\omega_{X_{k}} = in_{k}^{\ast}\big(\omega_{X}\big)$ uniquely determine a morphism
\begin{equation*}
   \omega_{X} \colon \pi_{1}(X) \longrightarrow \raisebox{0.5ex}{$\mathbb{Z}$}/\raisebox{-0.5ex}{$2\mathbb{Z}$}
\end{equation*}

Given  left $\mathbb{Z}[\pi_{1}(X_{k})]$-modules $M$, $N$ and a $\mathbb{Z}[\pi_{1}(X)]$-morphism $\alpha \colon M \longrightarrow N$, put
\begin{align}\label{dfn:thefunctorL}
\begin{split}
   L_{k}M &:= \Lambda \otimes_{\mathbb{Z}[\pi_{1}(X_{k})]} M \\
   L_{k}\alpha & := id_{\Lambda } \otimes _{\mathbb{Z}[\pi_{1}(X_{k})]}  \alpha \colon \Lambda\otimes_{\mathbb{Z}[\pi_{1}(X_{k})]} M \longrightarrow \Lambda \otimes_{\mathbb{Z}[\pi_{1}(X_{k})]} N
   \end{split}\tag{$\ast$}
\end{align}

Plainly,  $L_{k}$ is a functor from the category of left $\mathbb{Z}[\pi_{1}(X_{k})]$-modules to the category  of left $\Lambda$-modules.

Let $\mathbf{C}(\widetilde{X})$ be the $\mathbb{Z}[\pi_{1}(X)]$-cellular chain complex of the universal cover of $X$, and $\mathbf{B}$ the subcomplex containing the $n$-cell attached by $f_{1} + f_{2}$.

Then $\mathbf{B}$ is a Poincar{\'{e}} duality chain complex (\cite{HJB-BB2010} p.2361) and it follows Theorem 2.3 of \cite{BB2010} that $ L_{k}\big(\mathbf{C}(\widetilde{X}_{k})\big)$ is also a Poincar{\'{e}} duality chain complex.

Let $x$ be the chain representing the $n$-cell attached by $f_{1} + f_{2}$. Repeated application of Theorem 2.3 of \cite{BB2010} shows that $L_{1}\big(\mathbf{C}(\widetilde{X}_{1})\big) + L_{2}\big(\mathbf{C}(\widetilde{X}_{2})\big)$ is a Poincar{\'{e}} duality chain complex.

Hence $(X, \omega_{X}, [1 \otimes x])$ is a Poincar{\'{e}} duality complex. This is the \emph{connected sum of $(X_{1}, \omega_{X_{1}}, [X_{1}])$ and $(X_{2}, \omega_{X_{2}}, [X_{2}])$}, introduced by Wall in \cite{W1967}.

We have seen that a necessary condition for a Poincar{\'{e}} duality complex to be a  connected sum of Poincar{\'{e}} duality complexes is that its fundamental group be the free product of groups. We show that, in our context,  this condition is also sufficient.

\begin{mainb}
A PD$^{n}$ complex with $(n-2)$-connected universal cover, decomposes as  a connected sum if and only if its fundamental group decomposes as a free product of groups.
\end{mainb}

\begin{proof}
   It only remains to prove sufficiency.
   
   Let $(X, \omega_{X}, [X])$ be a PD$^{n}$ complex with $(n-2)$-connected universal cover such that 
   \begin{equation*}
     \pi_{1}(X) = G_{1} \ast G_{2}
   \end{equation*}
with $G_{1}, G_{2}$ groups.

As $\pi_{1}(X)$ is finitely presentable, so are $G_{1}$ and $G_{2}$.

For $j = 1,2$, let   $K_{j} = K(G_{j};1)$ be an Eilenberg-Mac Lane space with finite 2-skeleton. 

Then $K_{1} \vee K_{2}$ is an Eilenberg-Mac Lane space, $K(G_{1} \ast G_{2};1)$, and
\begin{equation*}
    H_{n} (K; \mathbb{Z}^{\omega}) = H_{n}(K_{1}; \mathbb{Z}^{\omega_{1}}) \oplus H_{n}(K_{n};\mathbb{Z}^{\omega_{2}})
\end{equation*}
where $\omega_{j} \in H^{1}(K_{j}; \raisebox{0.3ex}{$\mathbb{Z}$}/\raisebox{-0.2ex}{$2\mathbb{Z}$}) $ is the restriction of the orientation character $ \omega \in H^{1}(K; \raisebox{0.3ex}{$\mathbb{Z}$}/\raisebox{-0.2ex}{$2\mathbb{Z}$})$.

Thus, $\mu_{X} = \mu_{1} + \mu_{2}$, with $\mu_{j} \in H_{n} (K_{j} ; \mathbb{Z}^{\omega_{j}})$ for $j = 1,2$.

By the discussion above, if the PD$^{n}$ complex $X_{j}$, with $(n-2)$-connected universal cover realises the fundamental triple $(G_{j}, \omega_{j}, \mu_{j})$, then the connected sum of $X_{1}$ and $X_{2}$ realises the fundamental triple of $X$, whence, by the Classification Theorem in \cite{HJB-BB2010}, $X$ is orientedly homotopy equivalent to $X_{1} \# X_{2}$.

Hence, it is sufficient to construct realisations of $(G_{j}, \omega_{j}, \mu_{j}) \ (j=1,2).$ 

Let $L_{j}$ be the functor defined in (\ref{dfn:thefunctorL}), so that for $ i \ge 1$
\begin{equation*}
    C_{i}(\widetilde{K}) = L_{1}\big(C_{i}(\widetilde{K}_{1})\big) \oplus L_{2}\big(C_{i}(\widetilde{K}_{2})\big)
\end{equation*}

It follows that 
\begin{equation*}
   F^{n-1}\big(\mathbf{C}(\widetilde{K})\big) = L_{1}\Big(F^{n-1}\big(\mathbf{C}(\widetilde{K}_{1})\big)\Big) \oplus L_{2}\Big(F^{n-1}\big(\mathbf{C}(\widetilde{K}_{2})\big)\Big)
\end{equation*}
and
\begin{equation*}
  I \big(\pi_{1}(X)\big) = L_{1}\big(I(G_{1})\big) \oplus L_{2} \big(I(G_{2})\big)
\end{equation*}
where the canonical inclusion is given by
\begin{equation*}
    L_{j}\big(I(G_{j})\big) \longrightarrow I(G), \quad \mu \otimes \lambda \longmapsto \mu\lambda
\end{equation*}
for $\mu \in \mathbb{Z}[\pi_{1}(X)]$ and $\lambda \in I(G)$ viewed as an element of $I\big(\pi_{1}(X)\big)$.

Let \( \varphi_{j} \colon F^{n-1} \big(\mathbf{C}(\widetilde{K}_{i})\big) \rightarrowtail I(G_{j}) \) be a $\mathbb{Z}[G_{j}]$-morphism representing the class $ \nu_{C(\widetilde{K}_{j}), n-1}(\mu_{j})$.

Then the class, $ \nu_{C(\widetilde{K}, i-1}(\mu)$ of homotopy equivalences is represented by
\[
  \begin{CD}
  L_{1}\Big(F^{n-1}\big(\mathbf{C}(\widetilde{K}_{1}) \big) \Big) \oplus L_{2}\Big(F^{n-1}\big(\mathbf{C}(\widetilde{K}_{2}) \big) \Big)& \quad = \ F^{n-1}\big(\mathbf{C}(\widetilde{K}) \big) \\
  @VV{L_{1}(\varphi_{1}) \oplus L_{2}(\varphi_{2})}V\\
  L_{1}\big(I(G_{1}) \oplus L_{2} \big(I(G_{2})\big) &  = \  I(G)\phantom{ABC}
  \end{CD}
\] 
and it follows from the proof of the analogous proposition for $n = 3$ on pp.269--270 in \cite{VGT1990}, that $\varphi_{j}$ is a homotopy equivalence of modules. 

By Theorem A, $(G_{j}, \omega_{j}, \mu_{j})$ is realised by a PD$^{n}$ complex, $X_{j}$, with $(n-2)$-connected universal cover. 
\end{proof}

\end{document}